\documentclass[preprint]{elsarticle}
\pdfoutput=1 % if your are submitting a pdflatex (i.e. if you have
 \usepackage{mathrsfs}
\usepackage{physics}
\usepackage{amsthm}
\usepackage{mathtools}
\usepackage{hyperref}
\usepackage{xcolor}
\usepackage{enumerate}
\usepackage{amssymb}

\makeatletter
\def\ps@pprintTitle{%
	\let\@oddhead\@empty
	\let\@evenhead\@empty
	\let\@oddfoot\@empty
	\let\@evenfoot\@oddfoot
}
\makeatother

\renewcommand{\bar}{\overline}

\newcommand{\reals}{\mathbb{R}}
\newcommand{\comps}{\mathbb{C}}

\renewcommand{\Re}{\text{Re}}

\newcommand{\A}{\mathcal{A}}

\renewcommand{\H}{\mathcal{H}}

\newcommand{\B}{\mathcal{B}}

\newtheorem{theorem}{Theorem}
\newtheorem{corollary}[theorem]{Corollary}
\newtheorem{lemma}[theorem]{Lemma}
\newtheorem{prop}[theorem]{Proposition}
\newtheorem{definition}[theorem]{Definition}

\theoremstyle{definition}
\newtheorem{remark}[theorem]{Remark}

\numberwithin{theorem}{section}

\begin{document}
\begin{frontmatter}
\title{A short proof of Tomita's theorem}
\author[1]{Jonathan Sorce}
\ead{jsorce@mit.edu}
\affiliation[1]{organization={MIT Center for Theoretical Physics},%Department and Organization
	addressline={77 Massachusetts Ave}, 
	city={Cambridge},
	postcode={02139}, 
	state={MA},
	country={USA}}
\tnotetext[t1]{MIT Preprint ID MIT-CTP/5621}
\begin{abstract}
Tomita-Takesaki theory associates a positive operator called the ``modular operator'' with a von Neumann algebra and a cyclic-separating vector. Tomita’s theorem says that the unitary flow generated by the modular operator leaves the algebra invariant. I give a new, short proof of this theorem which only uses the analytic structure of unitary flows, and which avoids operator-valued Fourier transforms (as in van Daele's proof) and operator-valued Mellin transforms (as in Zsid\'{o}'s and Woronowicz’s proofs). The proof is similar to one given by Bratteli and Robinson in the special case that the modular operator is bounded.
\end{abstract}
\end{frontmatter}

\section{Introduction}

\textit{\textbf{Note:} This version of the paper differs from the journal version in that proposition \ref{prop:density} has been changed to a stronger statement, which is what was actually needed for the proof of corollary \ref{cor:tomitas-theorem}.
The original proof established the stronger statement.
I thank Richard Pieters for pointing this out to me in an email.
}

Let $\H$ be a Hilbert space, $\A$ a von Neumann algebra, and $\A'$ its commutant.
A vector $\Omega \in \H$ is said to be cyclic and separating for $\A$ if the subspaces $\A \Omega$ and $\A' \Omega$ are both dense in $\H$.
Given such a vector, one may define an antilinear operator $S_0$ with domain $\A \Omega$ by
\begin{equation}
	S_0 (\mathrm{a} \Omega) = \mathrm{a}^* \Omega, \qquad \mathrm{a} \in \A.
\end{equation}
The first basic result of the Tomita-Takesaki theory developed in \cite{takesaki2006tomita} --- see also \cite{takesaki-book, struatilua2019lectures} for textbook treatments --- is that $S_0$ is a preclosed operator.
Its closure is denoted $S$ and is called the \textit{Tomita operator}.
The polar decomposition of $S$ is written
\begin{equation}
	S = J \Delta^{1/2}.
\end{equation}
$J$ is an antiunitary operator called the \textit{modular conjugation}, and $\Delta$ is an invertible, positive, self-adjoint operator called the \textit{modular operator}.
One can also show that the Tomita operator of $\A'$ for $\Omega$ is $S^{*},$ with polar decomposition
\begin{equation}
	S^* = J \Delta^{-1/2}.
\end{equation}

Since the modular operator is invertible, it generates a unitary group of operators $\Delta^{-it}.$
The map from $\B(\H) \times \mathbb{R}$ to $\B(\H)$ given by
\begin{equation}
	g(\mathrm{x}, t) = \Delta^{-it} \mathrm{x} \Delta^{it}
\end{equation}
is known as \textit{modular flow}.
The fundamental theorem of Tomita, which is the starting point for the rest of the Tomita-Takesaki theory, is that modular flow maps the algebra $\A$ to itself.
The first complete proof of this theorem was given by Takesaki in \cite{takesaki2006tomita}.
Other general proofs were given in \cite{van-daele-proof, zsido-proof, woronowicz-proof}, and a simplified proof in the case that $\A$ is hyperfinite was given in \cite{longo-proof}.

Tomita-Takesaki theory has been extremely useful for the study of von Neumann algebras, especially the ones of type III.
It is the basic tool underlying Connes' classification of type III factors \cite{connes1973classification}, Takesaki's duality theorem for crossed products \cite{takesaki1973duality}, and also several interesting developments in quantum field theory \cite{Borchers:2000pv}.

This paper presents a proof of Tomita's theorem that is, to the best of my knowledge, new to the literature.
It is similar to Zsid\'{o}'s proof from \cite{zsido-proof} in that it works by constructing a dense subset of $\A$ for which modular flow admits an entire analytic extension.
The main difference is that for the dense subset constructed here, the analytic extension of modular flow has norm bounded by an exponential function at infinity, so that Carlson's theorem can be used to constrain modular flow by evaluating the analytic extension on the integers in the complex plane.
This lets us show directly that all commutators of the form
\begin{equation}
	[\Delta^{-it} \mathrm{a} \Delta^{it}, \mathrm{b}']\qquad \mathrm{b}' \in \A'
\end{equation}
vanish, which implies via von Neumann's bicommutant theorem the desired result
\begin{equation}
	\Delta^{-it} \mathrm{a} \Delta^{it} \in \A.
\end{equation}
By contrast, Zsid\'{o}'s proof proceeds using the theory of analytic generators \cite{cioranescu1976analytic}, which requires studying certain Mellin transforms of operator-valued functions.
(See also \cite{woronowicz-proof, Zsido2012}.)
Note also that the idea of using Carlson's theorem to constrain modular flow has appeared previously in \cite{bratteli2012operator} in the special case that the modular operator $\Delta$ is bounded.

Section \ref{sec:background} states some results from previous work that go into the new proof of Tomita's theorem.
Proofs are not given, but sources are provided.
Section \ref{sec:proof} presents a proof of Tomita's theorem.
A longer version of this proof with different exposition is presented in the companion article \cite{sorce-long-proof} for an audience of physicists.

\section{Background material}
\label{sec:background}

\begin{remark}
	The first important lemma tells us how to think of the domain of the Tomita operator $S$.
	The domain of $S_0$ is expressed in terms of operators in $\A$ acting on $\Omega.$
	The following lemma tells us that the domain of $S$ can be expressed in terms of operators affiliated to $\A$ acting on $\Omega.$
	(Recall that a closed, unbounded operator $\mathrm{T}$ is said to be affiliated to $\A$ if it commutes with every operator $\mathrm{a}' \in \A'$ on all vectors where both $\mathrm{T} \mathrm{a}'$ and $\mathrm{a}' \mathrm{T}$ are defined.)
\end{remark}

\begin{lemma} \label{lem:modular-domain}
	Let $\Omega$ be a cyclic-separating vector for a von Neumann algebra $\A,$ and let $S$ be the Tomita operator.
	The domain of $S$, which is the same as the domain of $\Delta^{1/2},$ consists of all vectors of the form $\mathrm{T} \Omega,$ where $\mathrm{T}$ is a closed operator affiliated with $\A$, having $\A' \Omega$ as a core, and for which $\Omega$ is in the domain of both $\mathrm{T}$ and $\mathrm{T}^{*}.$
	The operator $S$ acts as
	\begin{equation}
		S (\mathrm{T} \Omega) = \mathrm{T}^{*} \Omega.
	\end{equation}
\end{lemma}
\begin{proof}
	See e.g. \cite[lemma 1.10]{takesaki-book} or \cite[theorem 13.1.3]{jones2003neumann}.
\end{proof}

\begin{remark}
	The next lemma tells us when the unitary group generated by a positive operator, thought of as a flow acting on bounded operators by conjugation, can be analytically continued away from the imaginary axis in the complex plane.
\end{remark}

\begin{lemma} \label{lem:operator-continuation}
	Let $P$ be a positive, invertible, self-adjoint operator on the Hilbert space $\H.$
	Let $w$ be a complex number with $\Re(w) > 0.$
	Fix a bounded operator $\mathrm{x}.$
	
	If the operator $P^{-w} \mathrm{x} P^{w}$ is defined and bounded on a core for $P^{w}$, then for every $z$ in the strip $0 \leq \Re(z) \leq \Re(w),$ the operator $P^{-z} \mathrm{x} P^{z}$ is bounded on its domain, so that it is preclosed with bounded closure.
	The bounded-operator-valued function
	\begin{equation}
		z \mapsto \bar{P^{-z} \mathrm{x} P^{z}}
	\end{equation}
	is analytic in the strip with respect to the norm topology, and continuous on the boundaries of the strip with respect to the strong operator topology.
\end{lemma}
\begin{proof}
	See e.g. \cite[section 9.24]{struatilua2019lectures}.
\end{proof}

\begin{remark}
	The next lemma is due to Takesaki, and appears in all general proofs of Tomita's theorem except the one by Woronowicz.
	The reason for needing this lemma in the proof of section \ref{sec:proof} is that we will want to study a special class of operators $\mathrm{a} \in \A$ for which the modular operator $\Delta$ ``looks bounded.''
	Concretely, this will mean that the vector $\mathrm{a} \Omega$ lies in a spectral subspace of $\Delta$ with bounded spectral range. 
	To produce such vectors, we will start with other vectors in $\H$ and act on them with ``mollifying operators'' that truncate the large-spectral-value spectral subspaces of $\Delta$.
	Takesaki's lemma tells us what happens when the resolvent of the modular operator is used as a mollifier; other mollifying operators can be constructed from the resolvent using contour integrals.
\end{remark}

\begin{lemma} \label{lem:resolvent-lemma}
	Let $\Omega$ be a cyclic-separating vector for a von Neumann algebra $\A$ with commutant $\A'.$
	Let $\Delta$ be the associated modular operator.
	Let $z$ be in the resolvent set of $\Delta,$ so that $(z-\Delta)$ is invertible as a bounded operator.
	Fix $\mathrm{a}' \in \A'.$
	
	Then there exists a unique operator $\mathrm{a} \in \A$ satisfying
	\begin{equation}
		\mathrm{a} \Omega = (z - \Delta)^{-1} \mathrm{a}' \Omega,
	\end{equation}
	and it satisfies the bound
	\begin{equation}
		\lVert \mathrm{a} \rVert \leq \frac{\lVert \mathrm{a}' \rVert}{\sqrt{2 (|z| - \Re(z))}}.
	\end{equation}
\end{lemma}

\begin{remark}
	The final lemma we will need lets us express analytic functions of bounded operators as residue integrals.
\end{remark}
\begin{lemma}\label{lem:residue-lemma}
	Let $\mathrm{x}$ be a bounded, self-adjoint operator, and let $f$ be a function analytic in a neighborhood of the spectrum of $\mathrm{x}.$
	Then the operator $f(\mathrm{x})$, defined via the functional calculus using the spectral theorem, can be written in terms of the norm-convergent Bochner integral
	\begin{equation}
		f(\mathrm{x})
			= \frac{1}{2 \pi i} \int_{\gamma} dz\, f(z) (z - \mathrm{x})^{-1},
	\end{equation}
	where $\gamma$ is any simple, counterclockwise, closed contour in the domain of $f$ encircling the spectrum of $\mathrm{x}.$
\end{lemma}
\begin{proof}
	See e.g. \cite[sections 2.25 and 2.29]{struatilua2019lectures}.
\end{proof}

\section{A proof of Tomita's theorem}
\label{sec:proof}

Let $\Theta$ be the Heaviside theta function, defined by
\begin{equation}
	\Theta(x)
		=
		\begin{cases}
				1 & x > 0 \\
				\frac{1}{2} & x = 0 \\
				0 & x < 0
		\end{cases}.
\end{equation}
The main idea of this section is to produce operators in $\A$ for which the modular operator ``looks bounded'' by starting with a vector $\mathrm{a}' \Omega$ and acting on it with the operator $\Theta(\lambda - \Delta)$ for some $\lambda > 0.$
We will be able to study the vector $\Theta(\lambda - \Delta) \mathrm{a}'\Omega$ by approximating the function $\Theta(\lambda - x)$ with a sequence of sigmoid functions,
\begin{equation}
	f_k(x)
		= \frac{1}{1+e^{k (x - \lambda)}}.
\end{equation}
Since $f_k(z)$ is analytic in the complex plane, the operator $f_k(\Delta)$ (defined via the functional calculus from the spectral theorem) can be studied using a contour integral of $f_k(z)$ multiplied by the resolvent $(z - \Delta)^{-1}.$
Once the resolvent has entered our construction, we will be able to apply lemma \ref{lem:resolvent-lemma}.

\begin{figure}
	\centering
	\includegraphics{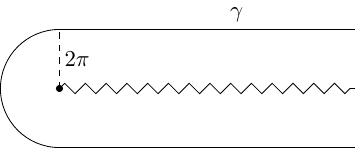}
	\caption{A sketch of the contour used in proposition \ref{prop:contour-prop} and theorem \ref{thm:bounded-construction}.
	The black dot denotes the origin of the complex plane, and the jagged line is the positive real axis.}
	\label{fig:specific-contour}
\end{figure}

\begin{prop} \label{prop:contour-prop}
	Fix $\lambda > 0,$ and let $f_k : \comps \to \comps$ be the sigmoid function
	\begin{equation}
		f_k(z)
			= \frac{1}{1 + e^{k(z-\lambda)}}. 
	\end{equation}
	Let $\Delta$ be an invertible, self-adjoint, positive operator. 
	Let $\gamma$ be the counterclockwise contour in the complex plane surrounding the positive real axis, given by combining the half-lines $\{t \pm 2 \pi i, \quad t \geq 0\}$ with the half-circle of radius $2 \pi i$ centered at the origin.
	(See figure \ref{fig:specific-contour}.)
	
	Then for any nonnegative integer $n,$ and any $\psi \in \H,$ we have
	\begin{equation}
		\Delta^n f_k(\Delta) \psi
			= \frac{1}{2 \pi i} \int_{\gamma} dz\, z^n f_k(z) (z - \Delta)^{-1} \psi,
	\end{equation}
	where this integral converges in the sense of the Bochner integral on Hilbert space.
\end{prop}
\begin{proof}
	For each fixed integer $m,$ let $v_m$ be the vertical segment passing through the real axis at $m+1/2$, oriented in the positive imaginary direction, and with endpoints on the contour $\gamma$ from figure \ref{fig:specific-contour}.
	Let $\gamma_m$ be the portion of the contour $\gamma$ lying to the left of this vertical segment.
	Let $\Pi_m$ be the spectral projection of $\Delta$ in the range $[0, m],$ and denote $\Delta_{(m)} \equiv \Delta \Pi_m$.
	
	Now, consider the vector $\Pi_m \psi.$
	We have
	\begin{align}
		\begin{split}
			\Delta^n f_k(\Delta) \Pi_m \psi
				= \Delta_{(m)}^n f_k(\Delta_{(m)}) \Pi_m \psi.
		\end{split}
	\end{align}
	Each $\Delta_{(m)}$ is bounded, so by lemma \ref{lem:residue-lemma}, we may express this equation in terms of a norm-convergent Bochner integral as
	\begin{align}
		\begin{split}
			\Delta^n f_k(\Delta) \Pi_m \psi
			= \frac{1}{2 \pi i} \int_{\gamma_m + v_m} z^n f_k(z) (z - \Delta_{(m)})^{-1} \Pi_m \psi.
		\end{split}
	\end{align}
	In fact, since the spectrum of $\Delta_{(m)}$ lies in the range $[0, m],$ we may write this integral for any $m' \geq m$ as 
	\begin{align}
		\begin{split}
			\Delta^n f_k(\Delta) \Pi_m \psi
			= \frac{1}{2 \pi i} \int_{\gamma_{m'} + v_{m'}} z^n f_k(z) (z - \Delta_{(m)})^{-1} \Pi_m \psi.
		\end{split}
	\end{align}
	The integral over $v_{m'}$ is easily seen to vanish in the limit $m' \to \infty,$ which gives the identity
	\begin{align}
		\begin{split}
			\Delta^n f_k(\Delta) \Pi_m \psi
				& = \frac{1}{2 \pi i} \int_{\gamma} z^n f_k(z) (z - \Delta_{(m)})^{-1} \Pi_m \psi \\
				& = \frac{1}{2 \pi i} \int_{\gamma} z^n f_k(z) (z - \Delta)^{-1} \Pi_m \psi \\
		\end{split}
	\end{align}

	So far we have shown that the proposition holds for any vector of the form $\Pi_m \psi.$
	But by the spectral theorem, the sequence $\Pi_m$ converges strongly to the identity operator.
	Taking the limit $m \to \infty$ in the above expression gives
	\begin{align}
		\begin{split}
			\Delta^n f_k(\Delta) \psi
			& = \frac{1}{2 \pi i} \lim_{m \to \infty} \int_{\gamma} z^n f_k(z) (z - \Delta)^{-1} \Pi_m \psi \\
		\end{split}
	\end{align}
	Applying the dominated convergence theorem lets us move the limit inside the integral, and proves the proposition.
 \end{proof}

\begin{prop} \label{prop:sigmoid-theta-prop}
	Fix $\lambda > 0,$ and let $f_k : \comps \to \comps$ be the sigmoid function
	\begin{equation}
		f_k(z) = \frac{1}{1 + e^{k(z-\lambda)}}.
	\end{equation}
	Let $\Delta$ be an invertible, self-adjoint, positive operator.
	Then for any $\psi \in \H$ and any nonnegative integer $n,$ the vector sequence $\Delta^n f_k(\Delta) \psi$ converges to $\Delta^n \Theta(\lambda - \Delta) \psi$ in the limit $k \to \infty.$
\end{prop}
\begin{proof}
	We aim to show the identity
	\begin{equation}
		\lim_{k \to \infty} \lVert \left(\Delta^n f_k(\Delta) - \Delta^n \Theta(\lambda - \Delta) \right) \psi\rVert = 0. 
	\end{equation}
	From the spectral theorem for $\Delta$, there exists a complex measure $\mu_{\psi}$ on $[0, \infty)$ satisfying
	\begin{equation}
		\lVert \left(\Delta^n f_k(\Delta) - \Delta^n \Theta(\lambda - \Delta) \right) \psi\rVert^2
			= \int d\mu_{\psi}(t)\, |t^n f_k(t) - t^n \Theta(\lambda - t)|^2.
	\end{equation}
	This can be seen to converge to zero by a standard application of the dominated convergence theorem.
\end{proof}

\begin{theorem} \label{thm:bounded-construction}
	Fix $\lambda > 0.$
	Let $\Delta$ be the modular operator of a cyclic-separating vector $\Omega$ for a von Neumann algebra $\A$ with commutant $\A'.$
	Let $n$ be a nonnegative integer, and fix $\mathrm{a}' \in \A'.$
	
	There exists a bounded operator $\mathrm{a}_{\lambda, n} \in \A$ satisfying
	\begin{equation} \label{eq:theta-switch}
		\mathrm{a}_{\lambda, n} \Omega = \Delta^n \Theta(\lambda - \Delta) \mathrm{a}' \Omega.
	\end{equation}
	Furthermore, there exist $n$-independent constants $\alpha_\lambda, \beta_\lambda > 0$ with $\lVert \mathrm{a}_{\lambda,n} \rVert \leq \alpha_\lambda e^{\beta_\lambda n}.$
	In particular, one can show the concrete bound
	\begin{align} \label{eq:tidy-bound}
	\begin{split}
		\lVert \mathrm{a}_{\lambda, n} \rVert
		& \leq \frac{\lVert \mathrm{a}' \rVert}{2 \pi} \left(2 \lambda \frac{(\lambda^2 + 4 \pi^2)^{n/2}}{\sqrt{2 ((\lambda^2 + 4 \pi^2)^{1/2} - \lambda)}} +  \frac{(2\pi)^{n+1} \pi}{\sqrt{4 \pi}} \right).
	\end{split}
\end{align}
\end{theorem}
\begin{proof}
	Since $\Delta^{n+1/2} \Theta(\lambda - \Delta)$ is a bounded operator, the vector $\Delta^n \Theta(\lambda - \Delta) \mathrm{a'} \Omega$ is in the domain of $\Delta^{1/2}.$
	So by lemma \ref{lem:modular-domain}, there exists a closed operator $\mathrm{a}_{\lambda, n}$ affiliated to $\A,$ with $\A' \Omega$ as a core, satisfying equation \eqref{eq:theta-switch}.
	The goal is to show that $\mathrm{a}_{\lambda, n}$ is bounded, and in fact that its norm is bounded by an exponential function of $n.$
	
	Since $\A' \Omega$ is a core for $\mathrm{a}_{\lambda, n},$ it suffices to show that $\mathrm{a}_{\lambda, n}$ has bounded action on vectors of the form $\mathrm{b}' \Omega$ for $\mathrm{b}' \in \A'.$
	Combining propositions \ref{prop:contour-prop} and \ref{prop:sigmoid-theta-prop}, and once again using $f_k$ to denote the sigmoid function from those propositions, we have
	\begin{align} \label{eq:first-integral-manipulations}
		\begin{split}
		\mathrm{a}_{\lambda,n} \mathrm{b}' \Omega
			& = \mathrm{b}' \mathrm{a}_{\lambda,n} \Omega \\
			& = \mathrm{b}' \Delta^n \Theta(\lambda - \Delta) \mathrm{a}' \Omega \\
			& = \mathrm{b}' \lim_{k \to \infty} \Delta^n f_k(\Delta) \mathrm{a}' \Omega \\
			& = \frac{1}{2\pi i} \mathrm{b}' \lim_{k \to \infty} \int_{\gamma} dz\, z^n f_k(z) (z - \Delta)^{-1} \mathrm{a}' \Omega,
		\end{split}
	\end{align}
	where $\gamma$ is the contour from figure \ref{fig:specific-contour}.
	By lemma \ref{lem:resolvent-lemma}, there exist operators $\mathrm{a}_z \in \A$ satisfying
	\begin{equation}
		(z - \Delta)^{-1} \mathrm{a}' \Omega
			= \mathrm{a}_z \Omega
	\end{equation}
	and
	\begin{equation} \label{eq:az-bound}
		\lVert \mathrm{a}_z \rVert
			\leq \frac{\lVert \mathrm{a}' \rVert}{\sqrt{2(|z| - \Re(z))}}.
	\end{equation}
	Since $\mathrm{b}'$ is a bounded operator, and since the integral in equation \eqref{eq:first-integral-manipulations} converges as a Bochner integral, we may move $\mathrm{b}'$ through the limit and through the integral symbol to write
	\begin{align}
		\begin{split}
			\mathrm{a}_{\lambda, n} \mathrm{b}' \Omega
			& = \lim_{k \to \infty} \frac{1}{2 \pi i} \int_{\gamma} dz\, z^n f_k(z) \mathrm{a}_z \mathrm{b}' \Omega.
		\end{split}
	\end{align}
	Taking norms on either side of the equation gives
	\begin{align}
		\begin{split}
			\lVert \mathrm{a}_{\lambda, n} \mathrm{b}' \Omega \rVert
			& \leq \frac{1}{2 \pi} \lVert \mathrm{b}' \Omega \rVert \limsup_{k \to \infty} \int_{\gamma} ds\, |z|^n |f_k(z)| \lVert \mathrm{a}_z \rVert,
		\end{split}
	\end{align}
	where $s$ is an arclength parameter for the contour $\gamma.$
	Using the bound \eqref{eq:az-bound}, we may write the inequality
	\begin{align} \label{eq:penultimate-inequality}
		\begin{split}
			\lVert \mathrm{a}_{\lambda, n} \mathrm{b}' \Omega \rVert
			& \leq \frac{\lVert \mathrm{a}' \rVert}{2 \pi} \lVert \mathrm{b}' \Omega \rVert \limsup_{k \to \infty} \int_{\gamma} ds\, \frac{|z|^n |f_k(z)|}{\sqrt{2 (|z| - \Re(z))}}.
		\end{split}
	\end{align}
	
	On either half-line $\{ t \pm 2 \pi i, \quad t \geq 0\},$ we have
	\begin{equation}
		f_k(t \pm 2 \pi i)
			= \frac{1}{1 + e^{k(t - \lambda)}}.
	\end{equation}
	This converges pointwise, in the limit $k \to \infty,$ to the Heaviside function $\Theta(\lambda - t).$
	An application of the dominated convergence theorem then gives
	\begin{align}
		\begin{split}
			\limsup_{k \to \infty} \int_{\text{half-line}} ds\, \frac{|z|^n |f_k(z)|}{\sqrt{2 (|z| - \Re(z))}}
			& = \int_{0}^{\lambda} dt\, \frac{(t^2 + 4 \pi^2)^{n/2}}{\sqrt{2 ((t^2 + 4 \pi^2)^{1/2} - t)}}.
		\end{split}
	\end{align}
	The integrand is monotonically increasing in $t,$ so the integral can be upper bounded by $\lambda$ times the value of the integrand at $t=\lambda,$ giving
	\begin{align} \label{eq:half-line-bound}
		\begin{split}
			\limsup_{k \to \infty} \int_{\text{half-line}} ds\, \frac{|z|^n |f_k(z)|}{\sqrt{2 (|z| - \Re(z))}}
			& \leq \lambda \frac{(\lambda^2 + 4 \pi^2)^{n/2}}{\sqrt{2 ((\lambda^2 + 4 \pi^2)^{1/2} - \lambda)}}.
		\end{split}
	\end{align}
	The other contribution to the contour integral in \eqref{eq:penultimate-inequality} is an integral over a half-circle, and may be written as
	\begin{equation}
		\int_{\text{half-circle}} ds\, \frac{|z|^n |f_k(z)|}{\sqrt{2 (|z| - \Re(z))}}
			= \int_{\pi/2}^{3 \pi/2} d\theta \frac{(2 \pi)^{n+1} |f_k(2 \pi e^{i \theta})|}{\sqrt{4 \pi (1 - \cos(\theta))}}.
	\end{equation}
	On this half-circle, the functions $f_k$ converge pointwise to $1,$ and another application of the dominated convergence theorem gives
	\begin{equation}
		\limsup_{k \to \infty} \int_{\text{half-circle}} ds\, \frac{|z|^n |f_k(z)|}{\sqrt{2 (|z| - \Re(z))}}
		= \int_{\pi/2}^{3 \pi/2} d\theta \frac{(2 \pi)^{n+1}}{\sqrt{4 \pi (1 - \cos(\theta))}}.
	\end{equation}
	On the half-circle, the denominator is lower-bounded by $\sqrt{4 \pi}$, which gives the simple approximation
	\begin{equation} \label{eq:half-circle-bound}
		\limsup_{k \to \infty} \int_{\text{half-circle}} ds\, \frac{|z|^n |f_k(z)|}{\sqrt{2 (|z| - \Re(z))}}
			\leq \frac{(2\pi)^{n+1} \pi}{\sqrt{4 \pi}}.
	\end{equation}
	Combining expressions \eqref{eq:half-circle-bound} and \eqref{eq:half-line-bound} with the expression \eqref{eq:penultimate-inequality}, we may bound each operator $\mathrm{a}_{\lambda, n}$ by
	\begin{align}
		\begin{split}
			\lVert \mathrm{a}_{\lambda, n} \rVert
			& \leq \frac{\lVert \mathrm{a}' \rVert}{2 \pi} \left(2 \lambda \frac{(\lambda^2 + 4 \pi^2)^{n/2}}{\sqrt{2 ((\lambda^2 + 4 \pi^2)^{1/2} - \lambda)}} +  \frac{(2\pi)^{n+1} \pi}{\sqrt{4 \pi}} \right).
		\end{split}
	\end{align}
\end{proof}

\begin{remark} \label{rem:tidy-construction}
	A completely symmetric argument to the one given above, with the substitutions $\Delta \leftrightarrow \Delta^{-1}$ and $\A \leftrightarrow \A',$ shows that for $\lambda > 0$ and $n \leq 0,$ and for $\mathrm{a} \in \A,$ there exists an operator $\mathrm{a}'_{\lambda, n} \in \A'$ satisfying
	\begin{equation}
		\mathrm{a}'_{\lambda, n} \Omega
			= \Delta^n \Theta\left( \lambda^{-1} - \Delta^{-1} \right) \mathrm{a} \Omega
			= \Delta^{n} \Theta(\Delta - \lambda) \mathrm{a} \Omega,
	\end{equation}
	and that the norm of $\mathrm{a}'_{\lambda, n}$ is bounded by an exponential function of $-n.$
	Specifically, it is bounded by
	\begin{align}
		\begin{split}
			\lVert \mathrm{a}'_{\lambda, n} \rVert
			& \leq \frac{\lVert \mathrm{a} \rVert}{2 \pi} \left(2 \lambda^{-1} \frac{(\lambda^{-2} + 4 \pi^2)^{-n/2}}{\sqrt{2 ((\lambda^{-2} + 4 \pi^2)^{1/2} - \lambda^{-1})}} +  \frac{(2\pi)^{-n+1} \pi}{\sqrt{4 \pi}} \right).
		\end{split}
	\end{align}

	Combining this observation with theorem \ref{thm:bounded-construction}, we may conclude that for any integer $n$ and any real numbers $\lambda_1, \lambda_2$ satisfying $0 < \lambda_1 < \lambda_2,$ there are operators $\mathrm{a}_{[\lambda_1, \lambda_2], n} \in \A$ and $\mathrm{a}_{[\lambda_1, \lambda_2], n}' \in \A'$ satisfying
	\begin{equation} \label{eq:tidy-equation}
		\mathrm{a}_{[\lambda_1, \lambda_2], n} \Omega = \mathrm{a}'_{[\lambda_1, \lambda_2], n} \Omega = \Delta^n \Theta(\lambda_2 - \Delta) \Theta(\Delta - \lambda_1) \mathrm{a} \Omega.
	\end{equation}
	Furthermore, the norm of either family of operators (the ones in $\A$ or the ones in $\A'$) is bounded by an exponential function of $|n|.$
\end{remark}

\begin{definition}
	The space of operators $\mathrm{a}_{[\lambda_1, \lambda_2], 0} \in \A$ obtained as in the preceding remark will be called the space of \textbf{tidy operators} in $\A,$ denoted $\A_{\text{tidy}}.$
\end{definition}

\begin{definition}
	If $\mathrm{a}$ is a tidy operator in $\A$, then we denote by $\mathrm{a}'$ the operator in $\A'$ satisfying
	\begin{equation}
		\mathrm{a} \Omega = \mathrm{a}' \Omega.
	\end{equation}
	For any integer $n,$ we denote by $\mathrm{a}_n$ and $\mathrm{a}'_n$ the operators in $\A$ and $\A'$ satisfying
	\begin{equation}
		\mathrm{a}_n \Omega = \mathrm{a}'_n \Omega = \Delta^n \mathrm{a} \Omega = \Delta^n \mathrm{a}'\Omega
	\end{equation}
\end{definition}

\begin{remark}
	The next lemma tells us how to think about the adjoints of tidy operators.
\end{remark}

\begin{lemma} \label{lem:dagger-ladder}
	Let $\mathrm{a}$ be a tidy operator in $\A,$ and $n$ an integer.
	We have
	\begin{equation}
		(\mathrm{a}'_{n+1})^{*} \Omega 
			= (\mathrm{a}_{n})^* \Omega.
	\end{equation}
\end{lemma}
\begin{proof}
	Fix $\mathrm{b} \in \A.$
	We have
	\begin{align}
	\begin{split}
		\langle (\mathrm{a}_{n+1}')^* \Omega, \mathrm{b} \Omega \rangle
			& = \langle \mathrm{b}^* \Omega, \mathrm{a}_{n+1}' \Omega \rangle \\
			& = \langle \mathrm{b}^* \Omega, \Delta \mathrm{a}_{n} \Omega \rangle \\
			& = \langle \mathrm{b}^* \Omega, S^* S \mathrm{a}_n \Omega \rangle \\
			& = \langle S \mathrm{a}_n \Omega, S \mathrm{b}^* \Omega \rangle \\
			& = \langle (\mathrm{a}_n)^{*} \Omega, \mathrm{b} \Omega \rangle.
	\end{split}
	\end{align}
	Since $\A\Omega$ is dense in $\H,$ this proves the lemma.
\end{proof}

\begin{prop} \label{prop:powers}
	For any integer $n$ and any tidy operator $\mathrm{a} \in \A_{\text{tidy}},$ the operator $\Delta^{n} \mathrm{a} \Delta^{-n}$ is defined and bounded on $\A_{\text{tidy}} \Omega,$ and on that subspace it is equal to $\mathrm{a}_{n}.$ 
\end{prop}
\begin{proof}
	Fix $\mathrm{b} \in \A_{\text{tidy}}.$
	By construction of $\A_{\text{tidy}},$ the vector $\mathrm{b} \Omega$ is in the domain of $\Delta^{-n}.$
	Note that from the expression
	\begin{align}
	\begin{split}
		\mathrm{a} \Delta^{-n} \mathrm{b} \Omega
			& = \mathrm{a} \mathrm{b}_{-n} \Omega,
	\end{split}
	\end{align}
	we see that $\mathrm{a} \Delta^{-n} \mathrm{b} \Omega$ is in the domain of the Tomita operator $S$.
	Applying lemma \ref{lem:dagger-ladder} gives
	\begin{align}
		\begin{split}
		S \mathrm{a} \Delta^{-n} \mathrm{b}\Omega
			& = S \mathrm{a} \mathrm{b}_{-n} \Omega \\
			& = \mathrm{b}_{-n}^* \mathrm{a}^* \Omega \\
			& = \mathrm{b}_{-n}^* (\mathrm{a}_{1}')^* \Omega \\
			& = (\mathrm{a}_1')^* \mathrm{b}_{-n}^* \Omega \\
			& = (\mathrm{a}_1')^* (\mathrm{b}_{-(n-1)}')^* \Omega.
		\end{split}
	\end{align}
	So $S \mathrm{a} \Delta^{-n} \mathrm{b} \Omega$ is in the domain of the Tomita operator for $\A',$ which is the adjoint $S^*.$
	This tells us that $\mathrm{a} \Delta^{-n} \mathrm{b} \Omega$ is in the domain of $S^* S = \Delta,$ and gives the identity
	\begin{align}
		\begin{split}
			\Delta \mathrm{a} \Delta^{-n} \mathrm{b}\Omega
			& = S^{*} S \mathrm{a} \Delta^{-n} \mathrm{b}\Omega \\
			& = S^{*} (\mathrm{a}_1')^* (\mathrm{b}_{-(n-1)}')^* \Omega \\
			& = \mathrm{b}_{-(n-1)}' \mathrm{a}_{1}' \Omega \\
			& = \mathrm{b}_{-(n-1)}' \mathrm{a}_1 \Omega \\
			& = \mathrm{a}_1 \mathrm{b}_{-(n-1)}' \Omega \\
			& = \mathrm{a}_1 \mathrm{b}_{-(n-1)} \Omega.
		\end{split}
	\end{align}
	Iterating this process $n$ times tells us that $\mathrm{b} \Omega$ is in the domain of $\Delta^n \mathrm{a} \Delta^{-n},$ and gives the identity
	\begin{align}
		\begin{split}
			\Delta^n \mathrm{a} \Delta^{-n} \mathrm{b}\Omega
			& = \mathrm{a}_{n} \mathrm{b} \Omega.
		\end{split}
	\end{align}
\end{proof}

\begin{prop} \label{prop:core}
	For any real number  $x,$ the space $\A_{\text{tidy}} \Omega$ is a core for $\Delta^{x}.$
\end{prop}
\begin{proof}
	To show that $\A_{\text{tidy}} \Omega$ is a core for $\Delta^{x},$ we must show that vectors of the form $\mathrm{a} \Omega \oplus \Delta^x \mathrm{a} \Omega,$ for $\mathrm{a}\in \A_{\text{tidy}},$ are dense in the graph of $\Delta^{x}.$
	Suppose that $\psi$ is a vector in the domain of $\Delta^{x}$ such that $\psi \oplus \Delta^{x} \psi$ is orthogonal to all such vectors.
	I.e., suppose that for all $\mathrm{a} \in \A_{\text{tidy}},$ we have
	\begin{align}
		\begin{split}
			0
			& = \langle \psi \oplus \Delta^{x} \psi, \mathrm{a} \Omega \oplus \Delta^x \mathrm{a} \Omega \rangle \\
			& = \langle \psi, \mathrm{a} \Omega \rangle + \langle \Delta^{x} \psi, \Delta^x \mathrm{a} \Omega \rangle \\
			& = \langle \psi, (1 + \Delta^{2x}) \mathrm{a} \Omega \rangle.
		\end{split}
	\end{align}
	Our goal is to show that whenever this expression is satisfied, $\psi$ vanishes.
	It suffices to show that the space $(1 + \Delta^{2x}) \A_{\text{tidy}} \Omega$ is dense in $\H.$
	
	To see this, note that per the construction of $\A_{\text{tidy}}$ from remark \ref{rem:tidy-construction}, each $\mathrm{a} \in \A_{\text{tidy}}$ satisfies an equation like
	\begin{equation}
		\mathrm{a} \Omega
		= \Theta(\lambda_2 - \Delta) \Theta(\Delta - \lambda_1) \mathrm{b} \Omega
	\end{equation}
	for some $0 < \lambda_1 < \lambda_2$ and some $\mathrm{b} \in \A.$
	The operator
	\begin{equation}
		(1 + \Delta^{2x}) \Theta(\lambda_2 - \Delta) \Theta(\Delta - \lambda_1)
	\end{equation}
	is bounded, injective, and invertible on the spectral subspace of $\Delta$ corresponding to the range $[\lambda_1, \lambda_2].$
	It therefore maps any dense subset of $\H$ into a dense subspace of that spectral subspace; in particular, the space
	\begin{equation}
		(1 + \Delta^{2x}) \Theta(\lambda_2 - \Delta) \Theta(\Delta - \lambda_1) \mathcal{A} \Omega
	\end{equation}
	is dense in the spectral subspace of $\Delta$ for the range $[\lambda_1, \lambda_2].$
	For any $0 < \lambda_1 < \lambda_2,$ this is a subspace of $(1 + \Delta^{2x}) \A_{\text{tidy}} \Omega.$
	Taking $\lambda_1 \to 0$ and $\lambda_2 \to \infty$ shows that $(1 + \Delta^{2x})\A_{\text{tidy}} \Omega$ is dense in the spectral subspace of $\Delta$ for the range $(0, \infty)$, which by invertibility of $\Delta$ is equal to all of $\H.$
\end{proof}

\begin{theorem}[Tomita's theorem for tidy operators] \label{thm:tomita}
	For any $\mathrm{a} \in \A_{\text{tidy}}$ and any $t \in \reals,$ the operator $\Delta^{-it} \mathrm{a} \Delta^{it}$ is in $\A.$
\end{theorem}
\begin{proof}
	Fix $\mathrm{b}' \in \A'.$
	By von Neumann's bicommutant theorem, it suffices to show
	\begin{equation}
		[\Delta^{-it} \mathrm{a} \Delta^{it}, \mathrm{b}'] = 0.
	\end{equation}
	
	Fix any integer $n \geq 0.$
	By proposition \ref{prop:powers}, the operator $\Delta^{-n} \mathrm{a} \Delta^{n}$ is defined on the dense subspace $\A_{\text{tidy}} \Omega,$ and is equal to $\mathrm{a}_{-n}$ on that subspace.
	By proposition \ref{prop:core}, $\A_{\text{tidy}} \Omega$ is a core for $\Delta^{n}.$
	Combining these observations with lemma \ref{lem:operator-continuation}, we see that for any complex $z$ in the strip with $0 \leq \Re(z) \leq n,$ the operator $\Delta^{-z} \mathrm{a} \Delta^{z}$ is bounded on its domain, and the map
	\begin{equation}
		F_{\mathrm{a}}(z) = \overline{\Delta^{-z} \mathrm{a} \Delta^z}
	\end{equation}
	is holomorphic in the interior of the strip and strongly continuous on the strip's boundary.
	Since this is true for any nonnegative integer $n$, it is true as a statement about the entire right half-plane.
	
	The function
	\begin{equation}
		F_{\mathrm{a} \mathrm{b}'}(z) = [\overline{\Delta^{-z} \mathrm{a} \Delta^z}, \mathrm{b}']
	\end{equation}
	is holomorphic in the right half-plane and strongly continuous on the imaginary axis, with norm bounded by
	\begin{equation} \label{eq:last-bound}
		\lVert F_{\mathrm{a} \mathrm{b}'}(z) \rVert
			\leq 2 \lVert \overline{\Delta^{-\Re(z)} \mathrm{a} \Delta^{\Re(z)}} \rVert  \lVert \mathrm{b}' \rVert.
	\end{equation}
	Recall also that for any integer $n,$ we have
	\begin{equation}
		\overline{\Delta^{-n} \mathrm{a} \Delta^n}
			= \mathrm{a}_{-n}.
	\end{equation}
	From this we observe that $F_{\mathrm{a} \mathrm{b}'}(n)$ vanishes for any nonnegative integer $n.$
	We also know by the Phragm\'{e}n-Lindel\"{o}f principle, inequality \eqref{eq:last-bound}, and remark \ref{rem:tidy-construction} that the norm of $F_{\mathrm{a} \mathrm{b}'}$ is bounded in imaginary directions, and bounded by an exponential function along the real axis.
	Applying Carlson's theorem gives $F_{\mathrm{a} \mathrm{b}'}(z) = 0,$ and putting $z=it$ proves the theorem.
\end{proof}

\begin{prop} \label{prop:density}
	We have $\A_{\text{tidy}}' = \A'.$
\end{prop}
\begin{proof}
	The inclusion $\A_{\text{tidy}} \subseteq \A$ gives $\A_{\text{tidy}}' \supseteq \A'.$
	We must show the reverse inclusion, $\A_{\text{tidy}}' \subseteq \A'.$
	
	To see this, recall that by proposition \ref{prop:core}, $\A_{\text{tidy}} \Omega$ is a core for $\Delta^{1/2},$ and hence for the Tomita operator $S$.
	Consequently, for any $\mathrm{a} \in \A,$ there exists a sequence of tidy operators $\mathrm{a}_n$ satisfying both
	\begin{equation}
		\lim_{n \to \infty} \mathrm{a}_n \Omega = \mathrm{a} \Omega
	\end{equation}
	and
	\begin{equation}
		\lim_{n \to \infty} \mathrm{a}_n^{*} \Omega = \mathrm{a}^{*} \Omega.
	\end{equation}
	Clearly, we also have for any $\mathrm{b}' \in \A'$ the limits
	\begin{equation}
		\lim_{n \to \infty} \mathrm{a}_n \mathrm{b}' \Omega = \mathrm{a} \mathrm{b}' \Omega
	\end{equation}
	and
	\begin{equation}
		\lim_{n \to \infty} \mathrm{a}_n^{*} \mathrm{b}' \Omega = \mathrm{a}^{*} \mathrm{b}' \Omega.
	\end{equation}

	Now, suppose $O$ is an operator in $\A_{\text{tidy}}',$ $\mathrm{a}$ is an operator in $\A,$ and $\mathrm{a}_n$ is a sequence of tidy operators converging as in the above equations.
	Fix $\mathrm{b}', \mathrm{c}' \in \A'.$
	We have
	\begin{align}
		\begin{split}
			\langle [O, \mathrm{a}] \mathrm{b}' \Omega, \mathrm{c}' \Omega \rangle
				& = \langle \mathrm{a} \mathrm{b}' \Omega, O^{*} \mathrm{c}' \Omega \rangle
						- \langle O \mathrm{b}' \Omega, \mathrm{a}^* \mathrm{c}' \Omega \rangle \\
				& = \lim_{n \to \infty} \left( \langle \mathrm{a}_n \mathrm{b}' \Omega, O^{*} \mathrm{c}' \Omega \rangle
				- \langle O \mathrm{b}' \Omega, \mathrm{a}_n^* \mathrm{c}' \Omega \rangle \right) \\
				& = \lim_{n \to \infty} \langle [O, \mathrm{a}_n] \mathrm{b}' \Omega, \mathrm{c}' \Omega \rangle \\
				& = 0.
		\end{split}
	\end{align}
	Since $\A' \Omega$ is dense in $\H,$ this implies that the commutator $[O, \mathrm{a}]$ vanishes.
	So every element of $\A_{\text{tidy}}'$ commutes with every element of $\A,$ as desired.
\end{proof}

\begin{corollary}[Tomita's theorem] \label{cor:tomitas-theorem}
	For any $\mathrm{a} \in \A$ and any $t \in \mathbb{R}$, the operator $\Delta^{-it} \mathrm{a} \Delta^{it}$ is in $\A.$
\end{corollary}
\begin{proof}
	Let $\mathrm{b}'$ be a tidy operator for the commutant algebra $\A',$ and fix arbitrary $\psi, \xi \in \H.$
	Consider the commutator
	\begin{align}
	\begin{split}
		\langle [\Delta^{-it} \mathrm{a} \Delta^{it}, \mathrm{b}'] \psi, \xi \rangle 
			& = \langle \Delta^{-it} \mathrm{a} \Delta^{it} \mathrm{b}' \psi, \xi \rangle 
				- \langle \mathrm{b}' \Delta^{-it} \mathrm{a} \Delta^{it} \psi, \xi \rangle 
	\end{split}
	\end{align}
	By theorem \ref{thm:tomita} applied to tidy operators of $\A',$ the operator
	\begin{equation}
		\mathrm{b}'(t) \equiv \Delta^{it} \mathrm{b}' \Delta^{-it}
	\end{equation}
	is in $\A'.$
	This gives
	\begin{align}
		\begin{split}
			\langle [\Delta^{-it} \mathrm{a} \Delta^{it}, \mathrm{b}'] \psi, \xi \rangle 
			& = \langle \mathrm{a} \mathrm{b}'(t) \Delta^{it} \psi, \Delta^{it} \xi \rangle 
			- \langle \mathrm{b}'(t) \mathrm{a} \Delta^{it} \psi, \Delta^{it} \xi \rangle \\
			& = \langle [\mathrm{a}, \mathrm{b}'(t)] \Delta^{it} \psi, \Delta^{it} \xi \rangle \\
			& = 0.
		\end{split}
	\end{align}
	So $\Delta^{-it} \mathrm{a} \Delta^{it}$ commutes with every tidy operator in $\A'.$
	By proposition \ref{prop:density} (exchanging the role of $\A$ and $\A'$ in the statement of the proposition), it is in $\A.$
\end{proof}

\section*{Acknowledgments}
\noindent
I am grateful to Brent Nelson for comments on the first draft of this paper, and in particular for noticing an error that led to the addition of proposition \ref{prop:density}.
This work was supported by the AFOSR under award number FA9550-19-1-0360, by the DOE Early Career Award, and by the Templeton Foundation via the Black Hole Initiative.

\bibliographystyle{elsarticle-num-names}
\bibliography{bibliography}

\end{document}